\documentclass[12pt]{article}
\usepackage{amsmath,amssymb,amsthm,mathrsfs,hyperref,theoremref,lipsum,lineno,enumitem,lmodern,color}

\newcommand{\Prod}{\mathop{\prod}\limits}

\newcommand{\Lim}{\mathop{\lim}\limits}
\newcommand\mycom[2]{\genfrac{}{}{0pt}{}{#1}{#2}}

\def\dint{\displaystyle \int}

\theoremstyle{plain}
\newtheorem*{thm*}{Theorem}
\newtheorem{thm}{Theorem}[section]
\newtheorem{conj}{Conjecture}[section]
\newtheorem{lem}{Lemma}[section]

\theoremstyle{definition}

\numberwithin{equation}{section}

\title{On the Lang-Trotter conjecture for a class of non-generic abelian surfaces}

\author{Mohammed Amin Amri} 
 \newcommand{\Addresses}{{
  \bigskip

  Mohammed Amin.~Amri, \textsc{Higher School of Education and Training, Ibn Tofail University,Kenitra, Morocco}\par\nopagebreak
  \textit{E-mail address}, Mohammed Amin~Amri: \texttt{amri.amine.mohammed@gmail.com}

}}   
\begin{document}

\maketitle
\begin{abstract}
\sloppy In the present article, we formulate a conjectural uniform error term in the Chebotarev-Sato-Tate distribution for abelian surfaces $\mathbb{Q}$-isogenous to a product of not $\overline{\mathbb{Q}}$-isogenous non-CM-elliptic curves, established by the author in \cite[Theorem 1.1]{Amri22}.  As a consequence, we provide a direct proof to the generalized Lang-Trotter conjecture recently formulated and studied in \cite{Chen}. 
\end{abstract}
\section{Introduction}
The study of the traces of the images of Frobenius elements by Galois representations coming from arithmetic geometry is of intrinsic interest, due to the wealth of information that they encode about the arithmetic properties of the objects they come from these Galois representations. One of the most striking works in this direction is the beautiful probabilistic model developed by S. Lang and H. Trotter in \cite{Lang} which leads them to predict the asymptotic behaviour of the traces of the image of Frobenius elements by the Galois representation $\rho_{E,\ell}$ associated to a non-CM elliptic curve $E$ defined over $\mathbb{Q}$. More precisely, for an arbitrary integer $t\in\mathbb{Z}$ the Lang-Trotter conjecture states that
$$
\#\{p\le x\;|\; \mathrm{Tr}\, \rho_{E,\ell}(\mathrm{Frob}_{p})=t \}\sim C(E,t) \dfrac{\sqrt{x}}{\log x}
$$
where $C(E,t)$ is an explicit constant depending on the image of $\rho_{E,\ell}$. \\

Lately, several authors studied this problem in more general contexts, such as higher dimensional abelian varieties (for instance see \cite{Chen,CW,Cojocaru,Katz}). The primary  goal of this paper is to build bridges between the Chebotarev-Sato-Tate distribution (cf., \cite{Amri22}) and the generalized Lang-Trotter conjecture for abelian varieties and thus show that the depth of the problem is probably hidden in the structure of the error term in the aforementioned distribution. Specifically, by adapting the framework from \cite{GJ} we shall investigate the rate of convergence of the limiting distribution given by \cite[Proposition 3.1]{Amri22} in the special case of abelian surfaces $\mathbb{Q}$-isogenous to a product of not $\overline{\mathbb{Q}}$-isogenous non-CM elliptic curves and derive a conjectural error term (cf., \thref{conj}).  As a consequence we give a conditional proof to the generalized Lang-Trotter conjecture (cf., \thref{mainTh}) formulated and studied in \cite{Chen}. \\

The paper is organized as follows. In the first section, we discuss the matter of the independence of arithmetic distribution of Frobenius traces at finite and infinite places which leads us to state (see \thref{CST}) a particular case of \cite[Theorem 1.1]{Amri22}. We thus set up notation to be used throughout the paper. In the second section, we  examine the rate of convergence in the limiting distribution discussed in the first section, and we end up by formulating a conjectural uniform error term (see \thref{conj}). In the last section, we give a conditional proof to the generalized Lang-Trotter conjecture. 
\paragraph*{Notation and terminology.} Throughout this paper we shall adopt the following notations and terminology.\\
For a sequence $(s_m)_{m\ge0}$ we write $\Lim_{m \overset{\backsim}{\to}\infty} s_m$ to designate the limit of the sequence when $m$ goes to $\infty$ under the ordering by divisibility, to make it concert according to \cite{Cojocaru}, one may take 
$$
\lim_{m \overset{\backsim}{\to}\infty} s_m:=\lim_{n\to \infty} s_{m_n}\;\;\text{where}\;\; m_n:=\prod_{\ell\le n}\ell^{n}.
$$
For a set $S$ and functions $f, g: S\longrightarrow \mathbb{R}$. The notation $f=O(g)$ (or $f\ll g$) signifies that there exist a real number $A>0$ and a subset $S_0$ of $S$ such that $|f(x)|\le A|g(x)|$ for every $x\in S_0$.The constant $A$ is called the implied constant. We write $f\asymp g$ to signify that there exist real numbers $A,B>0$  and a subset $S_0$ of $S$ such that $A|g(x)|\le |f(x)|\le B |g(x)|$ for every $x\in S_0$.
\section{Arithmetic distribution}
Let $A/\mathbb{Q}$ be a principally polarized abelian surface isogenous to a product of not $\overline{\mathbb{Q}}$-isogenous non-CM elliptic curves. Let $m>1$ be an integer, the action of the Galois group $G_{\mathbb{Q}}:={\rm Gal}(\overline{\mathbb{Q}}/\mathbb{Q})$ on the $m$-torsion points $A[m]\subset A(\overline{\mathbb{Q}})$ gives rise to a mod-$m$ Galois representation 
$$
\bar{\rho}_{A,m} : G_{\mathbb{Q}}\longrightarrow {\rm GSp}_4(\mathbb{Z}/m\mathbb{Z}).
$$
By taking the inverse limit over all integers $m$ we obtain an adelic representation
$$
\hat{\rho}_{A} : G_{\mathbb{Q}}\longrightarrow \mathrm{GSp}_4(\hat{\mathbb{Z}}).
$$
Let $\ell$ be a prime, by taking the limit over all powers of $\ell$ torsion points we obtain the $\ell$-adic representation attached to $A$
$$
\rho_{A,\ell} : G_{\mathbb{Q}}\longrightarrow \mathrm{GSp}_4(\mathbb{Z}_\ell).
$$
Let $N_A$ be the product of primes of bad reduction for $A$. For $p\nmid N_{A}$, let ${\rm Frob}_p$ denote an arithmetic Frobenius at $p$ in $G_{\mathbb{Q}}$ and consider the $p$-Weil polynomial of $A$
$$
L_{A,p}(T):=\det(TI_4-\rho_{A,\ell}({\rm Frob}_p))=T^4+ a_{1,p}T^3+ a_{2,p}T^2+ p a_{1,p}T +p^2\in\mathbb{Z}[T].
$$
By a theorem of Tate we know that any root of $L_{A,p}(T)$ has absolute value equal to $\sqrt{p}$, it then follows
$$
|a_{1,p}|\le 4\sqrt{p}.
$$
Accordingly, we shall investigate the distribution of the sequence $\left(\frac{a_{1,p}}{4\sqrt{p}}\right)_{p}$ in $[-1,1]$ as $p$ varies over primes whose corresponding Artin symbols belong to a prescribed conjugacy class of the Galois group of some finite Galois extension of $\mathbb{Q}$. \\
Let $K_m/\mathbb{Q}$ be the finite Galois extension corresponding, via Galois theory, to the kernel of the representation $\bar{\rho}_{A,m}$, we let denote by $\mathscr{G}(m)$ its Galois group. Then we have
$$
\mathscr{G}(m)\cong G_{\mathbb{Q}}/{\rm Ker}\, \bar{\rho}_{A,m}\cong {\rm Im}\, \bar{\rho}_{A,m}.
$$
For $t\in\mathbb{Z}\setminus\{0\}$, we define the conjugacy class of $\mathscr{G}(m)$
$$
\mathscr{C}(m,t):=\{g\in \mathscr{G}(m)\;|\; {\rm Tr}(g)\equiv t\pmod m\}.
$$
Let $\sigma_p:=\left(\frac{K_m/\mathbb{Q}}{p}\right)$ be the Artin symbol of $p$ relative to the extension $K_m/\mathbb{Q}$. In view of the above isomorphisms, one may write $\bar{\rho}_{A,m}({\rm Frob}_p)=\sigma_p$. On the other hand, one has the congruence
$$
{\rm Tr}(\sigma_p)\equiv a_{1,p}\pmod m.
$$
Therefore
$$
\sigma_p\subseteq\mathscr{C}(m,t)\Longleftrightarrow a_{1,p}\equiv t\pmod m.
$$
The distribution of $\left(\frac{a_{1,p}}{4\sqrt{p}}\right)_{p,\sigma_p\subset \mathscr{C}(m,t)}$ in $[-1,1]$ is governed by the Chebotarev-Sato-Tate group $ST$ of the abelian surface $A$ namely ${\rm SU}(2)\times{\rm SU}(2)\times \mathscr{G}(m)$. We let $\mu$ denote the Haar measure on $ST$, that is, the product measure of the Haar measure $\mu_{ST}$ on the Sato-Tate group $ST_{A}:={\rm SU}(2)\times{\rm SU}(2)$ of $A$, and the discrete measure on $\mathscr{G}(m)$. According to \cite[Table 5]{Fite} the density function of the image of $\mu_{ST}$ under the map 
$$
\begin{array}{ccc}
  X & \longrightarrow & [-4,4]\times [-6,6]\\ 
\quad x_p &\longmapsto & \left(\frac{a_{1,p}}{\sqrt{p}},\frac{a_{2,p}}{p} \right)
\end{array}
$$
where $X:={\rm Conj}(ST_A)$ is the set of conjugacy classes of $ST_{A}$, is given by  
$$\rho(s,y)=\dfrac{1}{2\pi^2}\sqrt{\dfrac{(y-2s+2)(y+2s+2)}{s^2-4y+8}},$$ 
with support region 
$$
S:=\left\lbrace (s,y)\in\mathbb{R}^2\;|\; y\ge 2s-2,\;y\ge -2s-2,\; y\le\frac{1}{4}s^2+2\right\rbrace.
$$
Consequently, if we let $\mu_1$ being the pushforward measure of $\mu_{ST}$ with respect to the map 
$$
\begin{array}{ccccc}
  X & \longrightarrow & [-1,1]\times [-1,1] & \longrightarrow & [-1,1] \\ 
\quad x_p &\longmapsto & \left(\frac{a_{1,p}}{4\sqrt{p}},\frac{a_{2,p}}{6p} \right)&\longmapsto &\frac{a_{1,p}}{4\sqrt{p}}
\end{array}
$$
its density function is given by
$$
\Phi(s):=\dfrac{2}{\pi^2}\int_{R(4s)}\rho(4s,y){\rm d}y,
$$
where $R(s)$ denote the defining interval of $y$ imposed by the constraints of $S$.
In this setting the equidistribution result \cite[Theorem 1.1]{Amri22} yields the following.

\begin{thm}\thlabel{CST}
Let $A/\mathbb{Q}$ be an abelian surface isogenous to a product of not $\overline{\mathbb{Q}}$-isogenous non-CM elliptic curves. The sequence $\left(\frac{a_{1,p}}{4\sqrt{p}}\right)_{p,\sigma_p\in \mathscr{C}(m,t)}$ is equidistributed with respect to the pushforward measure of $\mu $. In particular
\begin{equation}\label{eq1}
\lim_{x\to\infty} \dfrac{\#\{p\le x : \frac{a_{1,p}}{4\sqrt{p}}\in I,\;\; a_{1,p}\equiv t\pmod m\}}{\pi(x)}= \dfrac{\#\mathscr{C}(m,t)}{\#\mathscr{G}(m)}\int_{I}\Phi(s){\rm d}s,
\end{equation}
for any subinterval $I\subset [-1,1]$.
\end{thm}
\section{Convergence rate of the Chebotarev-Sato-Tate distribution}
At this stage we investigate the rate of convergence in \eqref{eq1} when $t$ is fixed and $I$ a subinterval of $[-1,1]$ containing $0$. In other words, keeping the notation of the previous section, we would like to have a bound for the error function
$$
E(x,t,m,I):=\dfrac{\#\{p\le x : \frac{a_{1,p}}{4\sqrt{p}}\in I,\;\; a_{1,p}\equiv t\pmod m\}}{\pi(x)}- \dfrac{\#\mathscr{C}(m,t)}{\#\mathscr{G}(m)}\int_{I}\Phi(s){\rm d}s.
$$
In \cite{AT99} Akiyama and Tanigawa conjectured that for an elliptic curve without complex multiplication such an error term has a strict square-root accuracy (cf.,\cite[\S 1.3]{MAZUR}). It seems natural and plausible that this conjecture holds true in our context too.
\begin{conj} \thlabel{conj1} For $t,m$ and $I\subset[-1,1]$ fixed, we have
$$
E(x,t,m,I)=O(x^{-1/2+\varepsilon})\quad\text{as}\;x\to\infty
$$
for every $\varepsilon> 0$.
\end{conj}
It is enlightening to note that the above error term is far from being uniform neither with respect to $m$ nor the diameter $\delta(I)$ of $I$ (because the implied constant necessarily depend on $m$ and $\delta(I)$). For our purpose, we need an error term which enable us to pass to the limit as $m \overset{\backsim}{\to}\infty$. To this end, we shall follow the heuristic sketched in \cite{GJ}. The underlying idea behind the heuristic is to suppose that the dominant contribution of 
\begin{equation}\label{eq7}
\dfrac{\#\{p\le x : \frac{a_{1,p}}{4\sqrt{p}}\in I,\;\; a_{1,p}\equiv t\pmod m\}}{\pi(x)}
\end{equation}
should match the dominant contribution of 
\begin{equation}\label{eq8}
\dfrac{\#\mathscr{C}(m,t)}{\#\mathscr{G}(m)}\int_{I}\Phi(s){\rm d}s
\end{equation}
as $m \overset{\backsim}{\to}\infty$ or $\delta(I)\to 0$.\\
Let us start first by determining the dominant contribution of \eqref{eq8} as $m \overset{\backsim}{\to}\infty$. We shall prove the following lemma.
\begin{lem} Let $F_t(m):=\frac{m\#\mathscr{C}(m,t)}{\#\mathscr{G}(m)}$, then  the limit $\Lim_{m \overset{\backsim}{\to}\infty}F_t(m)$ exists and gives rise to a well defined function $F: \mathbb{Z}\setminus\{0\}\longrightarrow[0,\infty)$ and for all subinterval $I$ of $[-1,1]$ containing $0$ and satisfying $\delta(I)<\frac{1}{\sqrt{\log m}}$ we have
\begin{equation}\label{eq5}
\dfrac{\#\mathscr{C}(m,t)}{\#\mathscr{G}(m)}\int_{I}\Phi(s){\rm d}s=\Phi(0)F(t)\dfrac{\delta(I)}{m}+O\left(\dfrac{\delta(I)}{m\log m}\right).
\end{equation}
\end{lem}
\begin{proof}

By \cite[Theorem 6]{Serre} we know that ${\rm Im}\,\hat{\rho}_A$ is open in $G(\hat{\mathbb{Z}})$ where $G$ is the algebraic group
$$
G:={\rm GL}_2\times_{\det}{\rm GL}_2=\{(g_1,g_2)\in{\rm GL}_2\times{\rm GL}_2\;|\; \det(g_1)=\det(g_2)\}.
$$
It follows that there exists an integer $m\ge 1$ such that ${\rm Im}\,\hat{\rho}_A=\pi^{-1}({\rm Im}\,\bar{\rho}_{A,m})$ where $\pi : G(\hat{\mathbb{Z}})\longrightarrow G(\mathbb{Z}/m\mathbb{Z})$ is the natural projection. Let $m_{A}$ be the least such integer and set
$$
m_{A,t}:=m_{A}\prod_{\ell|m_{A}}\ell^{v_\ell(t)}.
$$
It has been proved by Chen et al in \cite{Chen} that the limit $\Lim_{m \overset{\backsim}{\to}\infty}F_t(m)$ exists and gives rise to a well defined function $F: \mathbb{Z}\setminus\{0\}\longrightarrow[0,\infty)$, given by
$$
F(t):= \lim_{m \overset{\backsim}{\to}\infty}F_t(m)=F_t(m_{A,t})\cdot\prod_{\ell\nmid m_{A}} F_{t}(\ell^{v_\ell(t)+1}).
$$
Moreover they calculate explicitly the universal factors, for $\ell\nmid m_{A}$
$$
F_{t}(\ell^{v_\ell(t)+1})=\left\{\begin{array}{cc}
\frac{\ell(\ell^4-\ell^3-2\ell^2+\ell+2)}{(\ell^2-1)^2(\ell-1)} & \text{if}\; \ell\nmid t \\
\frac{\ell^6 - \ell^4 - 3\ell^3 - \ell}{\ell^6 - 2\ell^4 + \ell^2}+\frac{E(\ell,t)+\delta_2(v_\ell(t))(\ell-1)}{4\ell^{e_\ell(t)+4}(\ell^2-1)^2}& \text{otherwise}
\end{array}\right.
$$
where $e_{\ell}(t):= \lfloor(v_{\ell}(t)-1)/2\rfloor$, $\delta_2(v_\ell(t)):=\left\{\begin{array}{cc}
1&\text{if}\;2|v_\ell(t)\\
0& \text{otherwise}
\end{array}\right.$ and $$E(\ell,t):=\frac{\ell^3(\ell+1)(\ell^{4e_{\ell}(t)+4}-1)+(\ell^4+\ell^3+2\ell^2+\ell+1)(\ell^{4e_{\ell}(t)}-1)}{(\ell^2+1)(\ell+1)}.$$
For $m$ a sufficiently large highly divisible number, one may write $m=m_{A,t}\Prod_{\mycom{\ell\nmid m_{A}}{\ell\le \ell_n}}\ell^{\kappa_\ell}$ where $\ell_n$ stands for the largest prime dividing $m$. Then we have $F_t(m)=F_t(m_{A,t})\Prod_{\mycom{\ell\nmid m_{A}}{\ell\le \ell_n}} F_{t}(\ell^{v_\ell(t)+1})$, and hence
$$F_t(m)-F(t)=F_t(m)\left(1-\prod_{\ell>\ell_n}F_t(\ell^{v_\ell(t)+1})\right).$$
Noting that, for $0<x<1$ we have $x<|\log(1-x)|$, and $F_t(\ell^{v_\ell(t)+1})=1+O\left(\frac{1}{\ell^2}\right)$. Then one has
$$
1-\prod_{\ell> \ell_n}F_t(\ell^{v_\ell(t)+1})<\sum_{\ell>\ell_{n}}|\log F_t(\ell^{v_\ell(t)+1})|\ll \sum_{\ell>\ell_n}\frac{1}{\ell^2}\ll \frac{1}{\ell_n}.
$$
Since $F_t(m)=O(1)$, we immediately obtain
$$
F_{t}(m)=F(t)+O\left(\frac{1}{\ell_n}\right).
$$
By virtue of \cite[Lemma 1]{Erdos} one has $\log m\asymp \ell_n $ which yields
\begin{equation}\label{eq4}
F_{t}(m)=F(t)+O\left(\frac{1}{\log m}\right).
\end{equation}
Next, we claim that $\Phi'(s)=O(s)$ as $s\to 0$. In fact, on one hand we have
 
\begin{eqnarray*}
\Phi(s)&=&\left\{\begin{array}{cc}
\dfrac{2}{\pi^{2}}\dint_{8s-2}^{4s^{2}+2}\rho(4s,y){\rm d}y  & 0\le s \le 1 
\\
 \dfrac{2}{ \pi^{2}}\dint_{-8s-2}^{4s^{2}+2}\rho(4s,y){\rm d}y  & -1\le s <0 
\end{array}\right.\\
&=& \left\{\begin{array}{cc}
\dfrac{32|s|}{3\pi^2}\left((s^2+1)E\left(1-\dfrac{1}{s^2}\right)-2K\left(1-\dfrac{1}{s^2}\right)\right) & s \in [-1,0)\cup(0,1]\\ 
\dfrac{32}{3\pi^2} & s = 0
\end{array}\right.\\
\end{eqnarray*}
where $E(s):=\dint_{0}^{1}\frac{\sqrt{1-st^{2}}}{\sqrt{1-t^2}}{\rm d}t$ and $K(s):=\dint_{0}^{1}\frac{{\rm d}t}{\sqrt{(1-t^2)(1-st^{2})}}$ are the complete elliptic integrals. On the other hand, by the mean value theorem there exists some $0<\alpha<1$ such that
\begin{eqnarray*}
|\Phi'(s)|&=& \dfrac{32}{\pi^2}\left(K\left(1-\frac{1}{s^2}\right)-s^2E\left(1-\frac{1}{s^2}\right)\right),\\
&=& \dfrac{32|s|}{\pi^2}\dint_{0}^{1}\dfrac{(1-s^2)\sqrt{1-t^2}}{\sqrt{(1-t^2)s^2+t^2}}{\rm d}t,\\
&=& \dfrac{32|s|}{\pi^{2}}\dfrac{1-s^2}{\sqrt{(1-\alpha^2)s^2+\alpha^2}}\dint_{0}^{1}\sqrt{1-t^2}{\rm d}t, \\
&\le& \dfrac{8}{\alpha\pi}|s|.
\end{eqnarray*}
It then follows from the above in conjunction with the mean value theorem that there exist $\hat{\xi},\xi\in I$ such that
\begin{equation}\label{eq2}
\dint_I \Phi(s){\rm d}s=\delta(I)\Phi(\xi)=\delta(I)(\Phi(0)+\Phi'(\hat{\xi})\xi)=\delta(I)\Phi(0)+O(\delta(I)^3).
\end{equation}
Combining \eqref{eq2} and \eqref{eq4}, we see
\begin{equation*}
\dfrac{\#\mathscr{C}(m,t)}{\#\mathscr{G}(m)}\int_{I}\Phi(s){\rm d}s=\Phi(0)F(t)\dfrac{\delta(I)}{m}+O\left(\dfrac{\delta(I)}{m\log m}\right)+O\left(\dfrac{\delta(I)^3}{m}\right)\;\text{as}\; m \overset{\backsim}{\to}\infty.
\end{equation*}
As the above formula holds uniformly with respect to $\delta(I)$, since $\delta(I)<\dfrac{1}{\sqrt{\log m}}$, then both $O$-terms in the right-hand side have the same order.
\end{proof}
Next, we proceed to determine the dominant contribution of \eqref{eq7} as $m \overset{\backsim}{\to}\infty$. Observe that
\begin{multline*}
\#\left\{p\le x : \frac{a_{1,p}}{4\sqrt{p}}\in I,\;\; a_{1,p}\equiv t\pmod m\right\}= \#\left\{ \dfrac{t^2}{16\delta(I)^2}<p\le x : \frac{a_{1,p}}{4\sqrt{p}}\in I,\right.\\ a_{1,p}\equiv t\pmod m \Bigg\}+\#\left\{p\le \dfrac{t^2}{16\delta(I)^2} : \frac{a_{1,p}}{4\sqrt{p}}\in I, a_{1,p}\equiv t\pmod m \right\}.
\end{multline*}
Since $\Lim_{m \overset{\backsim}{\to}\infty}\#\left\{p\le \frac{t^2}{16\delta(I)^2} : \frac{a_{1,p}}{4\sqrt{p}}\in I, a_{1,p}\equiv t\pmod m \right\}=0$, it follows that the first term in the right-hand side of the above displayed formula gives the dominant contribution and the second one can be seen as an error term. Moreover, by the prime number theorem, one has $$\dfrac{\#\left\{p\le \frac{t^2}{16\delta(I)^2} : \frac{a_{1,p}}{4\sqrt{p}}\in I, a_{1,p}\equiv t\pmod m \right\}}{\pi(x)}=O\left(\dfrac{t^2\log x}{x\delta(I)^{2}}\right).$$
As the main terms of \eqref{eq8} and \eqref{eq7} match as $m \overset{\backsim}{\to}\infty$, they will cancel each other and we will have
$$
E(x,t,m,I)=O\left(\dfrac{t^2\log x}{x\delta(I)^{2}}\right)+O\left(\dfrac{\delta(I)}{m\log m}\right).
$$
We retain only the $O$ term of the highest order. Thus to prevent the violation of \thref{conj1} we need to take $C_1(\log x)^2<m\le C_{2}\delta(I)\sqrt{x}$ for some constants $C_1$ and $C_2$.\\
The heuristic above substantiate the following conjecture.
\begin{conj}\thlabel{conj} Let $A/\mathbb{Q}$ be an abelian surface isogenous to a product of not $\overline{\mathbb{Q}}$-isogenous non-CM elliptic curves. There exist constants $C_1$ and $C_2$ depending on $t$ and the abelian surface $A$ such that as $I$ runs through subintervals of $[-1,1]$ containing $0$, we have
$$ \dfrac{\#\{p\le x : \frac{a_{1,p}}{4\sqrt{p}}\in I,\;\; a_{1,p}\equiv t\pmod m\}}{\pi(x)}=  \dfrac{\#\mathscr{C}(m,t)}{\#\mathscr{G}(m)}\int_{I}\Phi(s){\rm d}s+O\left(\dfrac{\delta(I)}{m\log m}\right)$$
uniformly in the region given by $\delta(I)<\frac{1}{\sqrt{\log m}}$, as long as $C_1(\log x)^2<m\le C_{2}\delta(I)\sqrt{x}$.
\end{conj}

\section{A conditional proof of the Lang-Trotter conjecture}
In what follow we keep all the notation introduced in the previous sections. For an abelian surface $A$ and an arbitrary integer $t\in\mathbb{Z}$, we consider the counting function 
$$
\pi_{A}(x,t):=\#\{p\le x\;|\; p\nmid N_{A}\; a_{1,p}=t\}.
$$
\begin{thm}\thlabel{mainTh}
Let $A/\mathbb{Q}$ be a polarized abelian surface isogenous to a product of not $\overline{\mathbb{Q}}$-isogenous non-CM elliptic curves, for which \thref{conj} holds.  Let $t\in\mathbb{Z}\setminus\{0\}$, we then have 
$$
\pi_{A}(x,t)\sim\dfrac{\Phi(0)}{2} F(t) \dfrac{\sqrt{x}}{\log x}\qquad\text{as} \quad x\to\infty.
$$
\end{thm}
\begin{proof}
Let $x\ge 2$ be a real number, and $m:=m(x)$ an integer valued function, such that
$$
 |t|x^{1/4} <m<\dfrac{\sqrt{x}}{\log x}.
$$ 
Next, consider the interval $I=\left(-\frac{m-|t|}{4\sqrt{x}},\frac{m-|t|}{4\sqrt{x}}\right)$ and notice that $\delta(I)=\frac{m-|t|}{2\sqrt{x}}$.
From the second inequality of the above displayed equation, one sees $\sqrt{\log m}<\log x$. Thus
$$\delta(I)<\frac{1}{\sqrt{\log m}}.$$
It then follows from \thref{conj} that
\begin{equation}\label{eq9}
\dfrac{\#\{p\le x : \frac{a_{1,p}}{4\sqrt{p}}\in I,\;\; a_{1,p}\equiv t\pmod m\}}{\pi(x)}= \dfrac{\#\mathscr{C}(m,t)}{\#\mathscr{G}(m)}\int_{I}\Phi(s)ds+O\left(\dfrac{1}{\sqrt{x}\log m}\right).
\end{equation}
On the other hand one has
\begin{multline*}
\#\left\{p\le x : \frac{a_{1,p}}{4\sqrt{p}}\in I,\;\; a_{1,p}\equiv t\pmod m\right\}= \#\left\{ \frac{x|t|^2}{(m-|t|)^2}<p\le x : \frac{a_{1,p}}{4\sqrt{p}}\in I,\right.\\ a_{1,p}\equiv t\pmod m \Bigg\}+\#\left\{p\le \frac{x|t|^2}{(m-|t|)^2} : \frac{a_{1,p}}{4\sqrt{p}}\in I, a_{1,p}\equiv t\pmod m \right\}.
\end{multline*}
Since, for a prime number $p$ such that $\frac{x|t|^2}{(m-|t|)^2}<p\le x$, 
$$
\frac{a_{1,p}}{4\sqrt{p}}\in I\;\;\text{and}\;\;a_{1,p}\equiv t\!\!\pmod m 
\Longleftrightarrow a_{1,p}=t,
$$
we then have
$$
\pi_A(x,t)=\pi_A\left(\dfrac{x|t|^2}{(m-|t|)^2},t\right)+\#\left\{ \dfrac{x|t|^2}{(m-|t|)^2}<p\le x : \frac{a_{1,p}}{4\sqrt{p}}\in I, a_{1,p}\equiv t\pmod m\right\}.
$$
From \cite[(1) Theorem 2.4.1]{Chen} and the lower bound of $m$, we have $$\pi_A\left(\dfrac{x|t|^2}{(m-|t|)^2},t\right)\ll \dfrac{\sqrt{x}}{\log^{8/7-\varepsilon}x},$$ for every $\varepsilon>0$. Thus
$$
\#\left\{p\le x : \frac{a_{1,p}}{4\sqrt{p}}\in I,\;\; a_{1,p}\equiv t\pmod m\right\}=\pi_A(x,t)+O\left(\dfrac{\sqrt{x}}{\log^{8/7-\varepsilon}x}\right).
$$
It then follows from the prime number theorem that
\begin{equation}\label{eq6}
\dfrac{\#\left\{p\le x : \frac{a_{1,p}}{4\sqrt{p}}\in I,\;\; a_{1,p}\equiv t\pmod m\right\}}{\pi(x)} =\\ \dfrac{\pi_A(x,t)}{\pi(x)}+O\left(\dfrac{1}{\sqrt{x}\log^{1/7-\varepsilon} x}\right).
\end{equation}
By multiplying both \eqref{eq6} and \eqref{eq9} by $\sqrt{x}$ in conjunction with \eqref{eq5} we find that 
$$
\dfrac{\Phi(0)}{2}F(t)\left(1-\dfrac{|t|}{m}\right)+O\left(\dfrac{1}{\log m}\right)= \dfrac{\sqrt{x}\cdot \pi_A(x,t)}{\pi(x)}+O\left(\dfrac{1}{\log^{1/7-\varepsilon} x}\right).
$$
By letting $x\to\infty$ and then $m \overset{\backsim}{\to}\infty$ we get 
$$
\lim_{x\to\infty}\dfrac{\sqrt{x}\cdot\pi_A(x,t)}{\pi(x)}=\dfrac{\Phi(0)}{2}F(t).
$$
Whence, by the prime number theorem
$$
\pi_A(x,t)\sim\dfrac{\Phi(0)}{2}F(t)\dfrac{\pi(x)}{\sqrt{x}}\sim\dfrac{\Phi(0)}{2}F(t)\dfrac{\sqrt{x}}{\log x}.
$$
\end{proof}
\subsection*{Acknowledgment}
I wish to thank Gabor Wiese for his many valuable comments.

\Addresses
\end{document}